\pdfoutput=1
\documentclass[conference]{IEEEtran}
\IEEEoverridecommandlockouts

\usepackage{eso-pic}
\usepackage{hyperref} 
\hypersetup{hidelinks}

\usepackage{cite}
\usepackage{amsmath,amssymb,amsfonts}
\usepackage{graphicx}
\usepackage{textcomp}
\usepackage{xcolor}
\usepackage{verbatim}
\def\BibTeX{{\rm B\kern-.05em{\sc i\kern-.025em b}\kern-.08em
    T\kern-.1667em\lower.7ex\hbox{E}\kern-.125emX}}

\usepackage{algorithm}
\usepackage{algpseudocode}
\usepackage{amsthm}

\usepackage{url}
\usepackage[anythingbreaks]{breakurl}

\author{\IEEEauthorblockN{1\textsuperscript{st} Justin Stevens}
\IEEEauthorblockA{\textit{Division of Computational and Data Sciences} \\
\textit{Washington University in St. Louis} \IEEEauthorrefmark{1}\\
St. Louis, MO, USA \\
justin@numbertheoryguy.com}
\and
\IEEEauthorblockN{2\textsuperscript{nd} Duncan Wilson}
\IEEEauthorblockA{\textit{Independent Researcher} \\
\textit{Independent Researcher}\\
San Francisco, CA, USA \\
duncanscottwilson@ymail.com}
\and
\IEEEauthorblockN{3\textsuperscript{rd} Richard Kelley}
\IEEEauthorblockA{\textit{Independent Researcher} \\
\textit{Independent Researcher}\IEEEauthorrefmark{2}\\
Washington, DC, USA \\
richard.kelley@gmail.com}}

\usepackage{asymptote}

\newtheorem{thm}{Theorem}
\newtheorem{lemma}[thm]{Lemma}
\newtheorem{cor}[thm]{Corollary}
\begin{document}

\title{The Maximum Number of Sets for 12 Cards is 14 \thanks{\IEEEauthorrefmark{1} Work done at the University of Alberta. \IEEEauthorrefmark{2} Work done at the University of Nevada, Reno.}}

\AddToShipoutPictureBG*{%
  \AtPageUpperLeft{%
    \put(0,-30){\makebox[\paperwidth]{\parbox{0.85\paperwidth}{\centering\scriptsize
      \copyright~2025 IEEE. Personal use of this material is permitted. Permission
      from IEEE must be obtained for all other uses, in any current or future media,
      including reprinting/republishing this material for advertising or promotional
      purposes, creating new collective works, for resale or redistribution to servers
      or lists, or reuse of any copyrighted component of this work in other works.\\[2pt]
      Accepted \& expanded version. Published in: J.~Stevens, D.~Wilson, and R.~Kelley,
      ``The Maximum Number of Sets for 12 Cards is 14,'' \emph{2025 IEEE Conference
      on Games (CoG)}, Lisbon, Portugal, 2025.
      DOI: \href{https://doi.org/10.1109/CoG64752.2025.11114382}{10.1109/CoG64752.2025.11114382}
    }}}}}
\maketitle

\begin{abstract}
We analyze a mathematical property of the popular card game SET\textregistered{}. We present a revised geometric and computer-assisted proof that the maximum number of sets for $12$ cards is $14$ using finite field geometry, combinatorics, and graph theory. We present several computer algorithms for finding or approximating the maximum number of sets. We show a complete computer proof that iterates over all possible board configurations. We use this method to compute the maximum number of sets with $4$ properties for a small number of cards, but it is generally too inefficient. With this method, we compute the maximum number of sets for $3$ properties for all possible numbers of cards. Finally, we compare this to an existing approximation algorithm for finding maximal sets.
\end{abstract}


\section{Introduction}

SET\textregistered{} is a card game with a deck of 81 cards. In the standard version of the game, each card has four properties, each of which is capable of taking on three distinct values. A \textbf{set} comprises three cards, each quality of which is \textit{all the same or all different}. For presentation's sake, each card has four properties: color, quantity, shape, and shading. The shapes are typically ovals, squiggles, and diamonds, while the colors are red, purple, or green, the quantity is $1$, $2$, or $3$, and the shading is solid, striped, or open.  
\begin{figure}[h]
    \centering
    \includegraphics[scale=0.33]{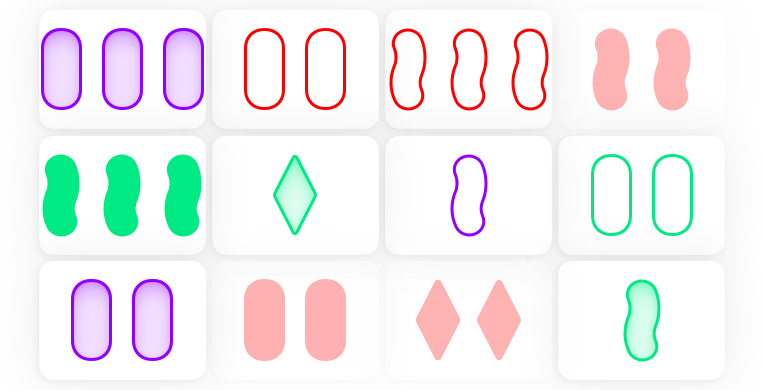}
\caption{An example board for the card game set. The opaque group of three cards form a set since the shapes are all different, while the color is all red, the quantity is all $2$, and the shading is solid.}
\label{fig:set}
\end{figure}

The set emphasized in Figure~\ref{fig:set} is a set because the shapes are all different, the color (red) is all the same, the shading is all the same, and the quantity is the same. The game is played with $12$ cards to create a \textbf{board}. Notice that the boards presented in this paper are in a $3\times 4$ rectangle, but all that matters is there are $12$ cards. 

The card game SET\textregistered{}, is played with multiple players who compete to find a set as fast as possible. 
This paper presents a novel proof of the maximum number of sets for the card game SET\textregistered{}. In exploring this proof, we also compare how computers can solve this problem using algorithms. We present two algorithms, a complete search and an existing approximation algorithm, consecutive maximization, with complete pseudocode. We have added the sequence of the maximum number of sets to the Online Encyclopedia of Integer Sequences in the hopes that future games and mathematical researchers can refer to it~\cite{oeis_A380167}. We also make our code available to researchers at \url{https://github.com/Duncanswilson/set-search}.

\section{Related Work} 
The SET\textregistered{} website introduces an algorithm to find a lower bound for the maximum number of sets for $n$ cards in general~\cite{vinci2009maximum} (called internal sets). The algorithm works by starting with an arbitrary pair of cards from different magic squares (referred to as cubes), selecting a third card to form a SET\textregistered{} with them, and then proceeding turn-by-turn, each time choosing the next card that maximizes the number of new SETs that can be formed with previously unused pairs of selected cards. We replicated the results and included them along with the pseudocode in the Appendix.

This problem is complementary to Knuth's SETSET-ALL problem \cite{Knuth}, where he finds the maximum number of cards with no set. Furthermore, a paper from MIT Primes conjectures the maximum number of quads (a variant of the card game set)~\cite{byrapuram2023maximum}, but does not have a full proof. Our work builds on prior work on the card game SET, including the book The Joy of Set~\cite{Joy}, geometric work~\cite{tucker2007geometric}, and the paper on the SET\textregistered{} website~\cite{vinci2009maximum}. Our work is also closely connected to work that attempts to solve larger versions of the card game set using $\alpha-\beta$ search~\cite{impartial-set}. Furthermore, our proof builds upon the combinatorics work of~\cite{benjamin2022proofs} and~\cite{brualdi2004introductory}.

This work is along the lines of several other work of formal proofs in games, including a paper that proves that Baba is You is Undecidable~\cite{geller2023baba} and that Checkers is Solved~\cite{schaeffer2007checkers}. The result for $12$ cards was proven previously~\cite{deBruijn2002Set}, and the authors discovered it after publication. In follow-up work, we extend the projection argument of Theorem~\ref{thm:subgraph} to determine the maximum number of sets for $13$, $14$, and $15$ cards.

\section{Problem Formulation}
Let all boards with $n$ cards and $k$ properties (with $3$ qualities per property) be $\mathcal{B}(n,k)$. Notice the maximum number of cards that can be on a board at a given time is $3^k$. This paper proves mathematically the maximum number of sets that can occur in an element of $\mathcal{B}(n,4)$ for $3\le n\le 12$. We also prove using complete search the maximum number of sets that can occur in an element of $\mathcal{B}(n,3)$ for $3\le n\le 27$. We finally present an approximation algorithm to compute the maximum number of sets for an element of $\mathcal{B}(n,4)$ for $13\le n\le 81$. 

\section{Geometric Proof for $3\le n\le 7$ cards}
\label{sec:points}

Given a card from $\mathcal{B}(n,4)$, we can choose an arbitrary order for the qualities and for each quality we can assign a numerical value of $0,1,$ or $2$. For example, one such assignment is shown in Table~\ref{tab:f3}. Using this assignment, the red card with two solid diamonds would be represented as the vector $(0,1, 2, 0)\in \mathbb{F}_3^4$. 

\begin{table}[h]
    \centering
       \caption{An example bijection for each quality of the card and the corresponding value in $\mathbb{F}_3$.}
    \begin{tabular}{cccc|c}
       Color & Quantity & Shape & Shading & $\mathbb{F}_3$ value\\ \hline 
        Red & $1$ & Oval & Solid & $0$ \\ 
       Purple & $2$ & Squiggles & Striped & $1$ \\
       Green & $3$ & Diamonds  & Open & $2$ 
    \end{tabular}
 
    \label{tab:f3}
\end{table}

Given such an assignment, there is a bijection between the elements of $\mathcal{B}(n,4)$ and a set $\{c_1, c_2, \cdots, c_n\}$ where each $c_i\in \mathbb{F}_3^4$. Observe that in proofs throughout the rest of the paper, we refer to the cards on a board as points in $\mathbb{F}_3^4$. 

Since the cards in a set must have all qualities either all the same or different, it follows that a set is a line in this field.

\begin{thm} If we have two points $p$ and $q$, for third point $r$ to be a set, it must satisfy: $p+q+r = 0$ (in $\mathbb{F}^4_3$). \label{thm:line} \end{thm}
\begin{proof}
We prove for each coordinate that the sum of the elements in $\mathbb{F}_3$ is $0$, so the sum is $0$ in $\mathbb{F}^4_3$. Observe that if all the qualities are the same, they must be $(0,0,0) $, $ (1,1,1)$, or $(2,2,2)$. In all three cases, the sum of the elements is $0$ in $\mathbb{F}_3$. If all three qualities are different, then their numerical values are $(0,1,2)$ (or a permutation thereof), which sum to $0$ in $\mathbb{F}_3$. In both cases, the sum of the numerical values is $0$, so $p+q+r=0$. 
\end{proof}

\begin{thm} The maximum number of sets for an element of $\mathcal{B}(3,4)$ and $\mathcal{B}(4,4)$ are both $1$. \end{thm}
\begin{proof} 
We begin by noting that $3$ cards are the minimum number to form a set, and with $3$ cards, there can be only one set. Next up, with $4$ cards, it is impossible for this fourth card to add an additional set, since if $\{p, q, r\}$ is a set, then $\{p, q, s\}$ cannot also form a set unless $r=s$.  \end{proof}

\begin{thm} The maximum number of sets for an element of $\mathcal{B}(5, 4)$ is $2$. \end{thm}
\begin{proof} 
For $5$ cards, we can take a $3$-card set, say $\{p, q, r\}$. Then, for two additional cards, $s, t$, outside of the set, the only possible way this can form a set is if we have one element from the first set plus $s$ and $t$ form a set, for example, $\{p, s, t\}$. An example of this is the points in $\mathbb{F}_3^4$, $$\{(0,0,0,0), (0, 0, 0, 1), (0, 0, 0, 2), (0, 0, 1, 0), (0, 0, 2, 0)\},$$ which form a maximum of $2$ sets for $5$ cards. Figure~\ref{fig:1} shows a visualization, where we simplify the representation by removing the first two coordinates (thus plotting it in $\mathbb{F}_3^2$).
\end{proof}

\begin{figure}[h]
\centering 
  \begin{asy}
    import graph;
    size(2inch);
    dot((0,0), 6+red);
    label("$(0,0)$", (0,0), SW);
    label("$(1,0)$", (1,0), SW);
    label("$(2,0)$", (2,0), SW);
    label("$(0,1)$", (0,1), SW);
    label("$(0,2)$", (0,2), SW);
    dot((1,0), 6+red);
    dot((0, 1), 6+red);
    dot((0,2), 6+red);
    dot((2,0), 6+red);
    draw((0,0)--(0,1)--(0,2), black);
    draw((0,0)--(1,0)--(2,0), black);
  \end{asy}
  \caption{A configuration of $5$ cards (points) that leads to two sets. Each point is represented as a dot in $\mathbb{F}_3^2$ since the first two elements of each card in the configuration are $0$. Observe that each of the lines is equivalent to a set in the graph.}
  \label{fig:1}
\end{figure}

\begin{thm} The maximum number of sets for an element of $\mathcal{B}(6, 4)$ is $3$. \end{thm} 
\begin{proof} 
For $6$ cards, we begin with taking a $3$-card set, again say $\{p, q, r\}$. Then we consider two cases for the remaining three cards, which we call $s, t,$ and $u$. If $\{s, t, u\}$ forms a set, then no new sets can form between the cards, since any new set would have to involve two cards from one set and one from the other. The next case to consider is if we have sets containing one element from the original set plus two elements outside this set. For example, we could have two new sets in $\{p, s, t\}$ and $\{q, t, u\}$. Notice that we cannot form any more sets this way, since if $\{r, s, u\}$ were also a set, we would have $p+s+t=q+t+u=r+s+u=0$. Summing up these equations gives us $$p+q+r+2(s+t+u)=0\implies s+t+u=0,$$ which would imply $\{s, t, u\}$ are a set. However, then $\{p, s, t\}$ could not be a set; for example, it would lead to a contradiction. This results in the following lemma.

\begin{lemma} \label{lem:1}
If $\{p, q, r\}$ is a set and $s, t, u$ are elements outside of this set, then it is impossible to form three sets using two elements from $s, t, u$ plus one element from the set $\{p, q, r\}$. \end{lemma} 
\begin{proof} The argument generalizes. Notice that this implies we cannot have edges leading to a set between $st$, $tu$, and $su$. Thus, no triangles can exist between these points.
\end{proof} 

 \noindent An example of $6$ cards (points in $\mathbb{F}_3^4$) forming $3$ sets is \begin{align*} \{(0, 0, 0, 0), (0, 0, 0, 1), (0, 0, 0, 2), \\ 
 (0, 0, 1, 0), (0, 0, 1, 1), (0, 0, 2, 0)\}. \end{align*}  This is shown in Figure~\ref{fig:2}, which we again show in $\mathbb{F}_3^2$: 
 \begin{figure}[h]

\centering 
  \begin{asy}
    import graph;
    size(2inch);
    dot((0,0), 6+red);
    dot((1,0), 6+red);
    dot((0, 1), 6+red);
    dot((0,2), 6+red);
    dot((2,0), 6+red);
    dot((1,1), 6+red);
    label("$(0,0)$", (0,0), SW);
    label("$(1,0)$", (1,0), SW);
    label("$(2,0)$", (2,0), SW);
    label("$(0,1)$", (0,1), SW);
    label("$(0,2)$", (0,2), SW);
    label("$(1,1)$", (1,1), SW);

    draw((0,0)--(0,1)--(0,2), black);
    draw((0,0)--(1,0)--(2,0), black);
    draw((0,2)--(1,1)--(2,0), black);
  \end{asy}
  \caption{A configuration of $6$ cards that lead to three sets. Notice among the three points $(1,1), (1,0)$ and $(2,0)$ there are no triangles.}
  \label{fig:2}
\end{figure}
\end{proof}
\begin{thm} The maximum number of sets for an element of $\mathcal{B}(7, 4)$ cards is $5$. \end{thm}
\begin{proof} 
For $7$ cards, we again begin with taking a $3$-card set, say $\{p, q, r\}$. Consider the remaining four cards not in this set; call them $s, t, u$, and $v$. Notice that if any three of them form a set, say $\{s, t, u\}$, then we could have $\{p, s, v\}$, $\{q, t, v\}$, and $\{r, u, v\}$ form sets. However, we cannot add any additional sets. Therefore, the maximum number of sets we found in this case is $5$. 

Suppose no three cards $s, t, u,$ and $v$ form a set. By Lemma~\ref{lem:1}, there can be no triangles (for example, an edge between $st$, $tu$, and $su$ in the graph of all sets). However, amongst $4$ points, the maximum possible number of edges without a triangle is $4$. The reason for this is that there are only $\binom{4}{2}=6$ possible edges, so if we remove one edge, say $st$, then taking the other two points plus one of these endpoints, in this case $tuv$, would still form a triangle. Thus, we cannot have $5$ new sets; in this case, the maximum number of new sets is also $4$.  \qedhere

An example of $7$ cards (points in $\mathbb{F}_3^4$) forming $5$ sets is $\{(0, 0, 0, 0), (0, 0, 0, 1), (0, 0, 0, 2), (0, 0, 1, 0),$ \newline
$(0, 0, 1, 1), (0, 0, 2, 0), (0, 0, 2, 1)\}.$  This is shown in Figure~\ref{fig:3}. 
 \begin{figure}[h]

\centering 
  \begin{asy}
    import graph;
    size(2inch);
    dot((0,0), 6+red);
    dot((1,0), 6+red);
    dot((0, 1), 6+red);
    dot((0,2), 6+red);
    dot((2,0), 6+red);
    dot((1,1), 6+red);
    dot((2, 1), 6+red);
    label("$(0,0)$", (0,0), SW);
    label("$(1,0)$", (1,0), SW);
    label("$(2,0)$", (2,0), SW);
    label("$(0,1)$", (0,1), SW);
    label("$(0,2)$", (0,2), SW);
    label("$(1,1)$", (1,1), SW);
    label("$(2,1)$", (2,1), SW);
    draw((0,0)--(0,1)--(0,2), black);
    draw((0,0)--(1,0)--(2,0), black);
    draw((0,2)--(1,1)--(2,0), black);
    draw((0,1)--(1,1)--(2,1), black);
    draw((1,0)..(2,1)..(0,2), black);
  \end{asy}
  \caption{A configuration of $7$ cards that lead to five sets. Each set is shown with a line or curve.}
  \label{fig:3}
\end{figure}

\end{proof} 

\section{A Geometric Proof with Magic Squares} 
A \textit{Magic Square} is a collection of nine cards, where each card is part of four valid sets\cite{Falco, Fairbanks}. Note that for any three non-collinear points (i.e., any three cards that do not form a set), we can construct a magic square, and that each magic square is a \textit{2-flat}, where a $k$-flat is defined by \cite{Davis} to be a $k$-dimensional affine subspace of a vector space, namely $\mathbb{F}^4_3$.\footnote{A set is a 1-flat, a magic square is a 2-flat, and a full hyperplane is a 3-flat.} 

\begin{thm} The maximum number of sets for an element of $\mathcal{B}(9, 4)$ is $12$. \end{thm}
\begin{proof} 
\label{F23}
If we hold two of the traits of the cards constant, we get a two-dimensional space over the finite field of three elements or an affine plane of order $3$. We know this space has $9$ points, since an affine plane of order $q$ has $q^2$ points \cite{Fairbanks}, i.e., $|\mathbb{F}^2_3|$.   Thus, there are $\binom{9}{2}$ possible lines, since two points determine a unique line of $3$ points. 

However, this value triple counts the number of lines as a set of points $\{p,q,r\}$; the pairs $\{p,q\}$,  $\{q,r\}$, and $\{p,r\}$ all determine the same line. So we have $\binom{9}{2} / 3 = 12$ lines. 

Figure~\ref{fig:magicsquare} also shows a configuration of $9$ cards/points that yields $12$ sets.
\begin{figure}[h]
\centering 
  \begin{asy}
    import graph;
    size(2inch);
    dot((0,0), 6+red);
    dot((1,0), 6+red);
    dot((0, 1), 6+red);
    dot((0,2), 6+red);
    dot((1, 1), 6+red);
    dot((1,2), 6+red);
    dot((2,1), 6+red);
    dot((2,0), 6+red);
    dot((2,2), 6+red);
    draw((0,0)--(0,1)--(0,2), black);
    draw((1,0)--(1,1)--(1,2), black);
    draw((0,1)--(1,1)--(2,1), black);
    draw((0,0)--(1,0)--(2,0), black);
    draw((0,2)--(1,1)--(2,0), black);
    draw((2,0)..(1,2)..(0,1), black);
    draw((0,0)--(2,1)--(1,2), grey);
    draw((1,0)..(2,1)..(0,2), black); 
    draw((2,0)--(2,1)--(2,2), black);
    draw((0,2)--(1,2)--(2,2), black);
    draw((0,0)--(1,1)--(2,2), black);
    draw((1,0)--(0,1)--(2,2), grey);
  \end{asy}
  \caption{The maximum number of sets for an element of $\mathcal{B}(9,4)$ is $12$. Each line or curve forms a set, as do the two grey segments.}
  \label{fig:magicsquare}
\end{figure}
\end{proof}
Throughout the rest of the proof, we call this graph the magic square or $G$. Notice that it does not necessarily need to be all of the points in $\mathbb{F}_3^2$, as you could, for example, have:

 \begin{align*}\{(0,0,0,0), (0, 1, 1, 0), (0, 2, 2, 0), (0, 0, 0, 1), (0, 1, 1, 1),  \\  (0, 2, 2, 1), (0, 0, 0, 2), (0, 1, 1, 2), (0, 2, 2, 2)\}. \end{align*}

\begin{cor}The maximum number of sets for an element of $\mathcal{B}(8, 4)$ is $8$. \end{cor} 
\begin{proof} 
We claim that removing a single point from the graph of $G$ in Figure~\ref{fig:magicsquare} results in a maximal configuration. In this case, every card is involved in exactly $3$ sets (which is maximal since we require $2$ other cards for a set and $\lfloor \frac{7}{2} \rfloor=3$). Since there are $8$ cards and $3$ are required for a set, this forms a total of $\frac{8\cdot 3}{3}=8$ sets. 
Figure~\ref{fig:4} shows the maximum number of sets for a board of $8$ cards.
\begin{figure}[h]

\centering 
  \begin{asy}
    import graph;
    size(2inch);
    dot((0,0), 6+red);
    dot((1,0), 6+red);
    dot((0, 1), 6+red);
    dot((0,2), 6+red);
    dot((1, 1), 6+red);
    dot((1,2), 6+red);
    dot((2,1), 6+red);
    dot((2,0), 6+red);
    draw((0,0)--(0,1)--(0,2), black);
    draw((1,0)--(1,1)--(1,2), black);
    draw((0,1)--(1,1)--(2,1), black);
    draw((0,0)--(1,0)--(2,0), black);
    draw((0,2)--(1,1)--(2,0), black);
    draw((2,0)..(1,2)..(0,1), black);
    draw((0,0)--(2,1)--(1,2), grey);
    draw((1,0)..(2,1)..(0,2), black); 
  \end{asy}
  \caption{The maximum number of sets for an element of $\mathcal{B}(8,4)$ is $8$. Each line or curve forms a set, as does the one grey segment.}
  \label{fig:4}
\end{figure}
\end{proof}

Observe that for $n\in\{10, 11, 12\}$, every maximum configuration of $\mathcal{B}(n,4)$ must contain $G$ as a subgraph; we prove this in Theorem~\ref{thm:subgraph} in Appendix~\ref{app:subgraph}.\footnote{This version corrects an error in the proof: the published version asserted this observation without a complete justification. Theorem~\ref{thm:subgraph} supplies the missing argument.} This observation immediately leads to the following corollary, too.

\begin{cor} The maximum number of sets for an element of $\mathcal{B}(10, 4)$ is $12$. \end{cor}
\begin{proof} Due to the magic square above being maximal, note that one extra card cannot contribute any extra sets to it. This is the same reasoning as going from $n=3$ to $n=4$, where we could not add any extra sets because $n=3$ has the maximum number of sets. Thus, the maximum number of sets for $n=10$ is the same as for $n=9$, which is $12$.  
\end{proof}
This also leads to the following third (less obvious) corollary of the magic square:
\begin{cor} The maximum number of sets for an element of $\mathcal{B}(11, 4)$ is $13$. \end{cor}
\begin{proof}
Based on our observation above, we must have a magic square with $9$ of the points forming $12$ sets. Therefore, we consider the two additional points not in the magic square, $\{p, q\}$. Notice that $p$ cannot be in a set with two elements from the magic square since the magic square is maximal. Thus, the only possible new set is one containing $p$ and $q$ together with one point from the magic square. This leads to a maximum of $13$ sets. A configuration in which $13$ sets are achieved is:

\begin{align*}\{(0,0,0,0), (0, 1, 0, 0), (0, 2, 0, 0), (1, 0, 0, 0), \\ (2, 0, 0,0), (1, 1, 0, 0), (2, 2, 0, 0), (1, 2, 0, 0), \\ (2, 1, 0, 0), (0, 0, 0, 1), (0, 0, 0, 2) \}. \end{align*}
\end{proof}

We now focus on the main proof: the maximum number of sets for $n=12$ cards. 
\begin{thm}\label{thm3} The maximum number of sets for an element of $\mathcal{B}(12, 4)$ is $14$. \end{thm}
We divide the proof into two sections: first, we construct a $14$-set board with $12$ cards, and then we prove that $15$ sets are impossible. 

\subsection{Construction of a 14 Set Board for $n=12$}
We claim that the $12$ points below in $\mathbb{F}_3^4$ form $14$ lines, yielding $14$ sets. 

\begin{table}[h]
\caption{A configuration of $\mathcal{B}(12,4)$ in $\mathbb{F}_3^4$ that lead to $14$ sets.}
\label{tab:14sets}
\centering 
\begin{tabular}{ c c c l} 
 $(0,2,1,0)$& $(0,2,0,0)$& $(0,1,0,0)$& $(0,0,0,0)$\\
 $(0,2,2,0)$& $(0,2,0,1)$& $(0,1,0,1)$& $(0,0,0,1)$\\ 
 $(0,0,2,0)$& $(0,2,0,2)$& $(0,1,0,2)$& $(0,0,0,2)$\\ \end{tabular}
\end{table} 

The last nine cards/three columns contain a magic square since the first and third elements equal zero. Thus, they form all of $\mathbb{F}_3^2$. This gives $12$ lines/sets. Furthermore, there are also two lines/sets using two of the points not in the magic square: 
\begin{align*} &\{(0,2, 1,0), (0,2, 0,0), (0,2,2,0)\} \text{ and } \\ 
& \{(0,2, 1, 0), (0,0, 2,0), (0,1,0,0)\}. \end{align*}
This gives $14$ lines/sets in total.

\begin{figure}[h]
\centering
\includegraphics[scale=0.3]{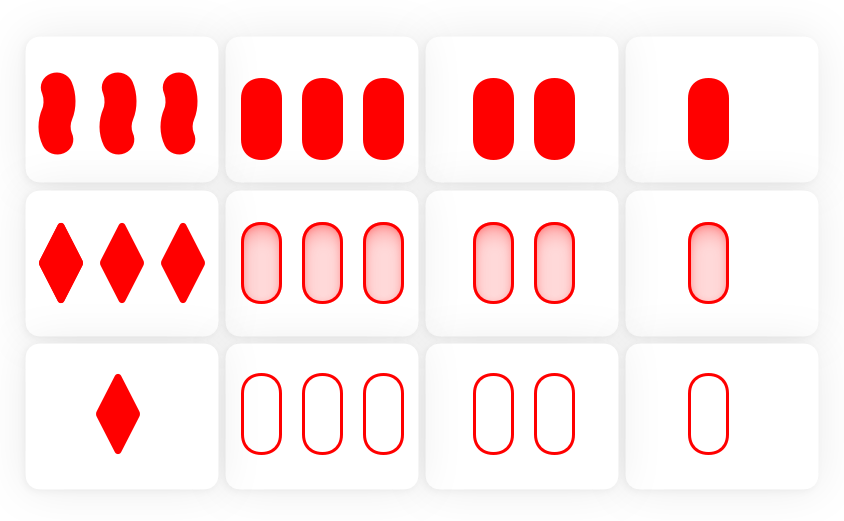}
\caption{Note this figure is isomorphic to Table~\ref{tab:14sets}. Can you find all $14$ sets on this board? }
\end{figure}

\subsection{Proof that $15$ sets is impossible for $n=12$}
\begin{proof} 
Suppose, for a contradiction, that the maximum number of SETs for 12 cards is at least 15. Let $S$ be a maximizing configuration, so $S$ contains at least $15$ SETs. By Theorem~\ref{thm:subgraph}, $S$ contains a complete magic square $G$. Translated back into geometric terms, we have at least 15 lines connecting 12 points in $\mathbb{F}^4_3$, and $9$ of these points form $G$. 

For any two points $p,q$ in $G$, the third point on the line through $p$ and $q$ is already in $G$, so we cannot form a line consisting of $p$ and $q$. Consider the three new points $s,t,u\not\in G$. Then, if we have 12 points connecting at least 15 lines, we must have at least three new lines formed by these points. Note that if $\{s, t, u\}$ were a line itself, then it is impossible to have $\{s, p, q\}$ be a line as explained before, and it is also impossible for $\{s, t, p\}$ to be a line, so we would have at most $13$ lines. Thus, we can only form three new lines if we combine two points not in $G$ plus one point in $G$. 

Assume without loss of generality that for $r\in G$ we can make three lines $\{p, s, t\}$, $\{q, t, u\}$ and $\{r, s, u\}$. If they did form lines, then looking at these as equations in $\mathbb{F}_3^4$, we get:
\begin{align*}
    p+s+t&=0 \\
    q+t+u&=0  \\
    r+s+u&=0. 
\end{align*}
Summing up these three equations, since they are symmetric, gives us:
\begin{align*}
    p+q+r+2(s+t+u) = 0
\end{align*}
From the first equation we have $s+t = -p$, therefore substituting this into the above equation gives:
\begin{align*}
    q+r-p+2u = 0 \implies 2u = p-q-r. 
\end{align*}
Since $p,q,r\in G$, and $G$ is fully connected, there must be some $z\in G$ such that $q, r, z$ form a line. Thus, $q+r+z=0$, so $ z=-q-r$. Furthermore, since $p,z\in G$, we see that there must exist some $y$ such that $p,z,y$ form a line, therefore $p+z+y=0$. Hence, substituting this into our equation above gives us:
\begin{align*} 2u &=p-q-r \\ &=p+z \\ &=-y \\ &=2y \tag{Since we're in $\mathbb{F}_3^4$} 
\end{align*}
However, then $u=y\in G$ contradicts the assumption that $u\not\in G$. So, at best, we can add two new lines, not three, and $15$ or more lines are impossible.  \end{proof} 
Therefore, we proved the paper's main result using finite field geometry over $\mathbb{F}_3^4$ and the magic square: the maximum number of sets for $12$ cards with $4$ properties is $14$. In doing so, we have proved the following values for the maximum number of sets with $4$ properties:
\begin{table}[h]
\centering 
    \caption{The maximum number of sets for an element of $\mathcal{B}(n,4)$ for $3\le n\le 12$ by geometric proof.}
    \begin{tabular}{c|cccccccccc}
    $n$ & $3$ & $4$ & $5$ & $6$ & $7$ & $8$ & $9$ & $10$ & $11$ & $12$ \\ \hline 
    Max $\#$ Sets & $1$ & $1$ & $2$ & $3$ & $5$ & $8$ & $12$ & $12$ & $13$ & $14$
    \end{tabular}
\end{table}

\section{Computer-Assisted Proofs}
Algorithm~\ref{alg:complete} shows the pseudocode for a computer-assisted proof of the maximum number of sets containing any element of $\mathcal{B}(n,k)$.
\begin{algorithm}
\caption{Complete computer-assisted proof for computing the maximum number of sets of an element of $\mathcal{B}(n,k)$}
\label{alg:complete}
\begin{algorithmic}[1]
\State $MaxSets \gets 0$ \label{line:init}
\ForAll{Configuration $b$ in $\mathcal{B}(n,k)$} \label{line:for} 
    \State $NumSets \gets 0$ \label{line:numsets}
    \ForAll{Combinations $c$ of $3$ cards in $b$} \label{line:comb}
        \If{IsSet(c)} \label{line:if}
            \State $NumSets \gets NumSets+1$ \label{line:increment}
        \EndIf
    \If{$NumSets>MaxSets$} \label{line:better} 
        \State $MaxSets \gets NumSets$ \label{line:update}
    \EndIf
    \EndFor
    \EndFor
\State \Return $MaxSets$ \label{line:return}

\end{algorithmic}
\end{algorithm}

In line~\ref{line:init}, we initialize the maximum number of sets to $0$. We then calculate the number of sets for each configuration $b$ beginning in line~\ref{line:for}. We start by checking each combination of $3$ cards in the board configuration $b$ in line~\ref{line:comb}. If these are a set, we then increment the number of sets in line~\ref{line:if} and line~\ref{line:increment}. If we find a new maximum number of sets, we update the number of sets in line~\ref{line:better} and line~\ref{line:update}. We finally return the maximum number of sets in line~\ref{line:return}.
\subsection{Computer Proof of the maximum number of sets for an element of $\mathcal{B}(n,4)$ for $3\le n\le 7$ cards}
We begin by using a computer search to prove the results above for the maximum number of sets for an element of $\mathcal{B}(n,4)$ with $3\le n\le 7$. We prove these using a complete computer search, where we enumerate every single element of $\mathcal{B}(n,4)$ and then compute the number of sets.

Table~\ref{tab:comp} shows the maximum number of sets and the corresponding computation time, using an Asus laptop with an AMD Ryzen 7 8845 Processor. All computer experiments used the same laptop processor to simplify comparisons. The results can be replicated by running the file \verb|set_searcher_complete.py| on GitHub\footnote{\url{https://github.com/Duncanswilson/set-search}}.

\begin{table}[h]
\centering 
    \caption{The maximum number of sets for $\mathcal{B}(n,4)$ for $3\le n\le 7$ by computer proof.}
    \begin{tabular}{c|c|l|l}
    $n$ & Max sets with $n$ cards & Compute Time (s) \\ \hline 
    $3$ & $1$ &  $<1$ \\
    $4$ & $1$ & $6$  \\
    $5$ & $2$ & $237$ \\
    $6$ & $3$ &  $5,611$ \\
    $7$ & $5$ & $105,120$
    \end{tabular}

    \label{tab:comp}
\end{table}

As computational time increases, continuing this way becomes impractical. In particular, for each configuration, we need to check $\tbinom{n}{3}$ sets. For $n=12$, this means checking $\binom{81}{12}\binom{12}{3}\approx 1.5\cdot 10^{16}$ possible sets, which would take an enormous amount of time to finish. 
\subsection{Computer Proof for $3$ properties}
If we limit the number of properties to $3$ instead of $4$, we can run a computer search to find the maximum number of sets. Notice that the number of triple checks for $n$ cards with $3$ properties is $\binom{27}{n}\cdot \binom{n}{3}$ since there are $\binom{27}{n}$ configurations to check, and for each of which we need to check $\binom{n}{3}$ possible sets. The results can be replicated by running the file \verb|set_search_3d_table.py|.
\begin{table}[h]
    \centering 
\caption{Maximum Number of Sets for an element of $\mathcal{B}(n,3)$ for $3\le n\le 27$.}
\begin{tabular}{c|c|c|l}
$n$ & Max Sets& Number of triple checks  & Compute Time (s) \\ \hline
$3$ & $1$ & $2.92\cdot 10^3$ & $0.00$ \\
$4$ & $1$ & $7.02\cdot 10^4$ & $0.06$ \\
$5$ & $2$ & $8.07\cdot 10^5$ & $0.64$ \\
$6$ & $3$ & $5.92\cdot 10^6$ & $4.64$ \\
$7$ & $5$ & $3.11\cdot 10^7$ & $23.76$ \\
$8$ & $8$ & $1.24\cdot 10^8$ & $94.78$ \\
$9$ & $12$ & $3.94\cdot 10^8$ & $301.44$ \\
$10$ & $12$ & $1.01\cdot 10^9$ & $761.40$ \\
$11$ & $13$ & $2.15\cdot 10^9$ & $1610.53$ \\
$12$ & $14$ & $3.82\cdot 10^9$ & $2872.26$ \\
$13$ & $16$ & $5.74\cdot 10^9$ & $4413.72$ \\
$14$ & $19$ & $7.30\cdot 10^9$ & $5322.08$ \\
$15$ & $23$ & $7.91\cdot 10^9$ & $5805.79$ \\
$16$ & $26$ & $7.30\cdot 10^9$ & $5243.79$ \\
$17$ & $30$ & $5.74\cdot 10^9$ & $4158.71$ \\
$18$ & $36$ & $3.82\cdot 10^9$ & $2850.46$ \\
$19$ & $41$ & $2.15\cdot 10^9$ & $1558.12$ \\
$20$ & $47$ & $1.01\cdot 10^9$ & $733.16$ \\
$21$ & $54$ & $3.94\cdot 10^8$ & $284.40$ \\
$22$ & $62$ & $1.24\cdot 10^8$ & $90.17$ \\
$23$ & $71$ & $3.11\cdot 10^7$ & $22.38$ \\
$24$ & $81$ & $5.92\cdot 10^6$ & $4.26$ \\
$25$ & $92$ & $8.07\cdot 10^5$ & $0.58$ \\
$26$ & $104$ & $7.02\cdot 10^4$ & $0.05$ \\
$27$ & $117$ & $2.92\cdot 10^3$ & $0.00$ \\
\end{tabular}

\label{tab:3prop}
\end{table}

Note that all values in this table are lower bounds for the maximum number of sets for $4$ properties, since we can set one property to be constant to get sets with $3$ properties. In the next section, we show another method that provides near-optimal lower bounds for the maximum number of sets and can compute it for $n>27$.

\section{Open Questions \& Conclusion} 

It remains an open question to find, with proof, the maximum number of sets with $4$ properties for $n>12$. We conjecture that the values in Table~\ref{tab:3prop} for $3$ properties match the maximum for $4$ properties. However, this only computes the table up to $n=27$. Therefore, further investigation is required for $n>27$. We presented Algorithm~\ref{alg:cmm} as a method for finding another lower bound for $n>27$. However, this method occasionally underestimates the actual value, as seen for $n=18$. Finding the maximum number of sets with more properties than $4$ also remains an open question. 

This paper presents a novel proof for finding the maximum number of sets for $n=12$ cards using finite field geometry. We also presented several algorithms for estimating or calculating the number of sets for $3$ and $4$ properties. We hope our work will inspire future researchers to pursue advanced mathematical or computational approaches to answer this question.

\section{Acknowledgements}
Thanks to Haydn Bradstreet for reviewing an early version of this paper. We also thank the anonymous reviewers for their many helpful comments and suggestions. We thank the team at the Online Encyclopedia of Integer Sequences for discussing the new sequence we discovered.

\bibliographystyle{IEEETran}
\bibliography{sample}

\newpage 

\appendices

\section{Consecutive Maximization Method}
\begin{table}
\centering 
\scriptsize
\renewcommand{\arraystretch}{0.97}
\caption{Approximate maximum number of sets for an element of $\mathcal{B}(n,4)$ computed using the consecutive maximization method~\cite{vinci2009maximum}.}
\label{tab:cmm}
\begin{tabular}{ccc}
$n$ & Approximated Maximum \# Sets & Compute Time (s) \\ \hline
$3$ & $1$ & 0.00\\
$4$ & $1$ & 0.00\\
$5$ & $2$ & 0.00\\
$6$ & $3$ & 0.00\\
$7$ & $5$ & 0.00\\
$8$ & $8$ & 0.00\\
$9$ & $12$ & 0.00\\
$10$ & $12$ & 0.00\\
$11$ & $13$ & 0.00\\
$12$ & $14$ & 0.01\\
$13$ & $16$ & 0.01\\
$14$ & $19$ & 0.01\\
$15$ & $23$ & 0.01\\
$16$ & $26$ & 0.02\\
$17$ & $30$ & 0.02\\
$18$ & $35$ & 0.03\\
$19$ & $41$ & 0.03\\
$20$ & $47$ & 0.03\\
$21$ & $54$ & 0.04\\
$22$ & $62$ & 0.04\\
$23$ & $71$ & 0.05\\
$24$ & $81$ & 0.06\\
$25$ & $92$ & 0.06\\
$26$ & $104$ & 0.07\\
$27$ & $117$ & 0.08\\
$28$ & $117$ & 0.09\\
$29$ & $118$ & 0.10\\
$30$ & $119$ & 0.11\\
$31$ & $121$ & 0.12\\
$32$ & $124$ & 0.13\\
$33$ & $128$ & 0.14\\
$34$ & $131$ & 0.15\\
$35$ & $135$ & 0.16\\
$36$ & $140$ & 0.17\\
$37$ & $146$ & 0.19\\
$38$ & $152$ & 0.21\\
$39$ & $159$ & 0.21\\
$40$ & $167$ & 0.23\\
$41$ & $176$ & 0.24\\
$42$ & $186$ & 0.25\\
$43$ & $197$ & 0.28\\
$44$ & $209$ & 0.29\\
$45$ & $222$ & 0.30\\
$46$ & $231$ & 0.32\\
$47$ & $241$ & 0.34\\
$48$ & $252$ & 0.36\\
$49$ & $264$ & 0.37\\
$50$ & $276$ & 0.38\\
$51$ & $289$ & 0.40\\
$52$ & $303$ & 0.42\\
$53$ & $318$ & 0.44\\
$54$ & $334$ & 0.46\\
$55$ & $351$ & 0.47\\
$56$ & $369$ & 0.49\\
$57$ & $388$ & 0.51\\
$58$ & $406$ & 0.53\\
$59$ & $425$ & 0.54\\
$60$ & $445$ & 0.56\\
$61$ & $466$ & 0.58\\
$62$ & $488$ & 0.60\\
$63$ & $511$ & 0.63\\
$64$ & $535$ & 0.63\\
$65$ & $560$ & 0.65\\
$66$ & $585$ & 0.66\\
$67$ & $611$ & 0.67\\
$68$ & $638$ & 0.70\\
$69$ & $666$ & 0.70\\
$70$ & $695$ & 0.71\\
$71$ & $725$ & 0.74\\
$72$ & $756$ & 0.75\\
$73$ & $788$ & 0.77\\
$74$ & $821$ & 0.77\\
$75$ & $855$ & 0.78\\
$76$ & $890$ & 0.80\\
$77$ & $926$ & 0.79\\
$78$ & $963$ & 0.84\\
$79$ & $1001$ & 0.85\\
$80$ & $1040$ & 0.81\\
\end{tabular}
\end{table}

In this section, we present the results mentioned in the related work section using the consecutive maximization method. Algorithm~\ref{alg:cmm} describes the pseudocode for the consecutive maximization algorithm as introduced in~\cite{vinci2009maximum}.

Table~\ref{tab:cmm} summarizes the results. The total time to run all of these is $24.91$ seconds, which is roughly the same as the complete search for $\mathcal{B}(7,3)$ in Table~\ref{tab:3prop}. 

Note that this algorithm is not always 100\% precise: it estimates $35$ sets for $\mathcal{B}(18,4)$, while Table~\ref{tab:3prop} shows that an element of $\mathcal{B}(18,3)$ can have $36$ sets.
As Theorem~\ref{tab:sets} shows, the maximum number of sets for $\mathcal{B}(18,3)$ is less than or equal to the maximum number of sets for $\mathcal{B}(18,4)$. Therefore, Table~\ref{tab:3prop} is also a lower bound on the number of sets for $\mathcal{B}(n,4)$ for $13\le n\le 27$. 
\begin{thm} \label{tab:sets} For all $i\le j$, we have that the maximum number of sets for $\mathcal{B}(n,i)$ is less than or equal to the maximum number of sets for $\mathcal{B}(n,j)$.\end{thm} 
\begin{proof} 
We see that an element of $\mathcal{B}(n,i)$ has $n$ cards with $i$ properties, while $\mathcal{B}(n,j)$ has $n$ cards with $j$ properties. If we hold $j-i$ properties constant, we get an element isomorphic to an element of $\mathcal{B}(n,i)$. On the other hand, for an arbitrary element of $\mathcal{B}(n,i)$, we can add $j-i$ empty properties to get to an element of $\mathcal{B}(n,j)$. Therefore, we can reduce an element of $\mathcal{B}(n,j)$ to every element of $\mathcal{B}(n,i)$, in particular, the maximum configuration of $\mathcal{B}(n,i)$. 
\end{proof} 
However, Algorithm~\ref{alg:cmm} helps to find near-optimal set configurations quickly. Furthermore, it also provides a lower bound to the true value of $\mathcal{B}(n, 4)$ for $n\ge 13$. 

\begin{algorithm}
\caption{Consecutive Maximization Method}
\label{alg:cmm}
\begin{algorithmic}[1]
\State $P = \text{ number of properties}$ \Comment{selected cards maximizing internal sets}

\State $deck \gets \text{GenerateDeck}(P)$
\State $selected \gets \emptyset$
\State $cubes \gets \text{GenerateCubes}(deck)$ \Comment{Organize into $3^{P-2}$ cubes}

\State Select initial cards from different cubes
\State $card1 \gets \text{SelectFirstCard}(cubes[1])$
\State $card2 \gets \text{SelectFirstCard}(cubes[2])$
\State $selected \gets selected \cup \{card1, card2\}$

\State Complete first Set
\State $card3 \gets \text{FindThirdCard}(card1, card2)$
\State $selected \gets selected \cup \{card3\}$

\For{$turn \gets 4$ to $3^P$}
    \If{$\exists t \in [1,P-1]: turn = 3t + 1$}
        \State $card \gets \text{SelectFromUnusedCube}(cubes, selected)$
    \Else
        \State $maxSets \gets 0$
        \State $bestCard \gets \text{null}$
        \ForAll{$c \in deck \setminus selected$}
            \State $newSets \gets \text{CountNewSets}(selected \cup \{c\})$
            \If{$newSets > maxSets$}
                \State $maxSets \gets newSets$
                \State $bestCard \gets c$
            \ElsIf{$newSets = maxSets \land \text{CubeOf}(c) \neq \text{CubeOf}(selected[\text{last}])$}
                \State $bestCard \gets c$ \Comment{Prefer different cube}
            \EndIf
        \EndFor
        \State $card \gets bestCard$
    \EndIf
    \State $selected \gets selected \cup \{card\}$
\EndFor
\State \Return $selected$
\end{algorithmic}
\end{algorithm}

We also detail the helper functions for Algorithm~\ref{alg:cmm}. 

\begin{itemize}
    \item The function CountNewSets will count the number of valid sets in a collection of cards by examining all possible pairs, finding a third card that would complete each set, and checking if that card exists in the collection. Since each set contains three pairs, it divides this count by $3$.
    \item The function FindThirdCard determines the third card needed to complete a set with two given cards. For matching properties, it keeps the same value; for different properties, it calculates a value that's different from both input cards. 
    \item GenerateDeck creates a deck of cards with $P$ properties, where each property can have values $1$, $2$, or $3$. 
    \item The function GenerateCubes will partition a deck into $3^{P-2}$ cubes, each containing $3\times 3$ cards. 
    \item The function SelectFirstCard will return the first available card from a given cube. 
    \item The function SelectFromUnusedCube will find the first available card from the first unused cube, that is, a cube that has no intersection with the cards already selected.
\end{itemize}

\begin{algorithm}[h]
\begin{algorithmic}[1]
\Function{CountNewSets}{$cards$}
    \State $count \gets 0$
    \ForAll{$(c1,c2) \in \text{Pairs}(cards)$}
        \State $c3 \gets \text{FindThirdCard}(c1, c2)$
        \If{$c3 \in cards$}
            \State $count \gets count + 1$
        \EndIf
    \EndFor
    \State \Return $count/3$ \Comment{each set is counted once per pair}
\EndFunction
\end{algorithmic} 
\end{algorithm}

\begin{algorithm}[h]
\begin{algorithmic}[1]
\Function{FindThirdCard}{$card1, card2$}
    \ForAll{property $p$}
        \If{$card1[p] = card2[p]$}
            \State $card3[p] \gets card1[p]$
        \Else
            \State $card3[p] \gets 6 - card1[p] - card2[p]$
        \EndIf
    \EndFor
    \State \Return $card3$
\EndFunction
\end{algorithmic}
\end{algorithm}

\begin{algorithm}[h]
    \begin{algorithmic}[1] 
\Function{GenerateDeck}{$P$}
    \State \Return $\{cards \text{ with } P \text{ properties, each } \in \{1,2,3\}\}$
\EndFunction
    \end{algorithmic} 
\end{algorithm}

\begin{algorithm}
    \begin{algorithmic}[1]
\Function{GenerateCubes}{$deck$}
    \State \Return partition of $deck$ into $3^{P-2}$ 3×3 cubes
\EndFunction
\end{algorithmic} 
\end{algorithm}

\begin{algorithm}[h]
    \begin{algorithmic}[1] 
    \Function{SelectFirstCard}{$cube$}
    \State \Return first available card from $cube$
\EndFunction
    \end{algorithmic} 
\end{algorithm}

\begin{algorithm}[h]
    \begin{algorithmic}[1] 
\Function{SelectFromUnusedCube}{$cubes, selected$}
    \ForAll{$cube \in cubes$}
        \If{$cube \cap selected = \emptyset$}
            \State \Return first available card from $cube$
        \EndIf
    \EndFor
\EndFunction
\end{algorithmic}
\end{algorithm}

\clearpage
\section{Proof that Maximum Configurations Contain a Magic Square}
\label{app:subgraph}
As in Section~\ref{sec:points}, we refer to the cards on a board as points in $\mathbb{F}_3^k$, so an element of $\mathcal{B}(n,k)$ is a set of $n$ distinct points of $\mathbb{F}_3^k$. By Theorem~\ref{thm:line}, every set $\{a,b,c\}$ satisfies $a+b+c=0$. Conversely, three distinct points with $a+b+c=0$ form a set: in each coordinate, if two of the three values are equal to some $x$, then the third value is $-2x=x$, so the values are all the same or all different. For a board $S$, let $N(S)$ be the number of sets in $S$, and let $M(n,k)$ be the maximum of $N(S)$ over all $S\in\mathcal{B}(n,k)$. We call $S$ a \emph{maximum} (or \emph{maximizing}) \emph{configuration} of $\mathcal{B}(n,k)$ if $N(S)=M(n,k)$. Two distinct points $a$ and $b$ lie on exactly one line $\{a,\,a+v,\,a+2v\}$, where $v=b-a$, and we say that this line has \emph{direction} $v$. A direction $v$ is \emph{used} by $S$ if two points of $S$ differ by a multiple of $v$, and \emph{unused} otherwise.

\begin{thm} Every maximizing configuration of $\mathcal{B}(n,4)$, for $n\in\{10,11,12\}$, contains a complete nine-point affine plane, that is, a magic square $G$. \label{thm:subgraph} \end{thm}
\begin{proof}
We begin by showing that it suffices to prove
this for $\mathcal{B}(n,3)$.

Consider the possible directions of lines in
$\mathbb{F}_3^4$. There are
\[
\frac{3^4-1}{2}=40
\]
possible directions, since the nonzero vectors
$v$ and $2v$ represent the same direction.

As we saw by construction above, there exist
configurations of 10, 11, and 12 cards
containing 12, 13, and 14 SETs, respectively.

Now consider any maximizing configuration $S$ of
$n$ cards containing $m=N(S)$ SETs. The number of
pairs of cards is $\binom{n}{2}$, but each SET
has three pairs determining the same line.
Thus, every SET reduces the number $L$ of distinct
lines through two points of $S$ by two, giving
\[
L=\binom{n}{2}-2m.
\]
Since every maximizing configuration contains
at least as many SETs as our constructions,
we have
\begin{align*}
n=10:\quad L&\leq\binom{10}{2}-2(12)=21<40,\\
n=11:\quad L&\leq\binom{11}{2}-2(13)=29<40,\\
n=12:\quad L&\leq\binom{12}{2}-2(14)=38<40.
\end{align*}

Every direction used by $S$ is the direction
of one of these $L$ lines.
Therefore, in every case there is at least one
unused direction $v$. Consider the projection
\[
\pi:\mathbb{F}_3^4\longrightarrow
\mathbb{F}_3^4/\langle v\rangle
\cong\mathbb{F}_3^3.
\]

Since $v$ is unused, no two selected points
differ by a multiple of $v$, and these multiples
are exactly the nonzero elements of the kernel of $\pi$.
Therefore, the projection preserves all $n$ distinct
points.

Furthermore, every original SET remains
a SET after projection, since its three images are distinct and
\[
a+b+c=0
\implies
\pi(a)+\pi(b)+\pi(c)=0.
\]

Since projection preserves
distinct points and all original SETs, we see that
\[
N(\pi(S))\geq N(S).
\]
However, any configuration in $\mathbb{F}_3^3$
can be embedded into $\mathbb{F}_3^4$ by
adding a zero coordinate. Since $S$ is
maximizing, we also have
\[
N(\pi(S))\leq N(S).
\]
Consequently,
\[
N(\pi(S))=N(S).
\]
The same embedding shows that
$M(n,3)\le M(n,4)=N(S)=N(\pi(S))\le M(n,3)$, so $\pi(S)$
is a maximizing configuration of $\mathcal{B}(n,3)$.

An exhaustive computer search, extending
Algorithm~\ref{alg:complete}, shows that
every maximizing configuration of
$\mathcal{B}(n,3)$ for $n\in\{10,11,12\}$
contains a complete magic square
(Table~\ref{tab:magic}; see below).

Therefore, $\pi(S)$ contains a magic square
consisting of nine points.

Since the projection is injective and the
number of SETs is unchanged, every SET
in this projected magic square must also
be a SET among its nine original preimages.

Thus, each of these nine original points is part of
four SETs among them, so they form a
complete magic square in $S$.

Therefore, every maximizing configuration
of $\mathcal{B}(n,4)$ for
$n\in\{10,11,12\}$ contains $G$, which is a
nine-point affine plane since each magic square is a 2-flat.
\end{proof}

\paragraph*{The computer search} The computational step in the proof of Theorem~\ref{thm:subgraph} can be replicated by running the file \verb|set_search_3d_magic_square.py| on GitHub, which extends Algorithm~\ref{alg:complete}. For each $n\in\{10,11,12\}$, the program enumerates the $n$-point configurations of $\mathbb{F}_3^3$, counts their SETs, checks whether they contain any of the $39$ affine planes of $\mathbb{F}_3^3$ (a complete magic square), and records the maximum number of SETs, $M(n,3)$, and the maximum number of SETs among the configurations that contain no magic square. Table~\ref{tab:magic} shows the results. For each $n$, the maximum without a magic square is less than $M(n,3)$, so every maximizing configuration contains a complete magic square.

\begin{table}[h]
\centering
\caption{The maximum number of SETs for an element of $\mathcal{B}(n,3)$.}
\label{tab:magic}
\begin{tabular}{c|c|c}
$n$ & $M(n,3)$ & Max, no magic square \\ \hline
$10$ & $12$ & $9$ \\
$11$ & $13$ & $10$ \\
$12$ & $14$ & $12$
\end{tabular}
\end{table}

\end{document}